\documentclass[11pt]{article}
\usepackage{amsmath}
\usepackage{amsthm}
\usepackage{amssymb}

\newtheorem{theo}{Theorem}[section]

\newtheorem{prop}[theo]{Proposition}
\newtheorem{lemma}[theo]{Lemma}

\newtheorem{conj}[theo]{Conjecture}

\newcommand{\FF}{{\cal F}}

\begin{document}
\date{}

\title{
On the product dimension of clique factors
}

\author{Noga Alon
\thanks
{Department of Mathematics, Princeton University,
Princeton, NJ 08544, USA and
Schools of Mathematics and
Computer Science, Tel Aviv University, Tel Aviv 69978,
Israel.  
Email: {\tt nogaa@tau.ac.il}.  
Research supported in part by
NSF grant DMS-1855464, ISF grant 281/17, GIF grant G-1347-304.6/2016,
and the Simons Foundation.
}
\and
Ryan Alweiss
\thanks
{Department of Mathematics, Princeton University,
Princeton, NJ 08544, USA.
Email: {\tt alweiss@math.princeton.edu}. Research supported by an NSF Graduate Research Fellowship.
}}

\maketitle
\begin{abstract}
The product dimension of a graph $G$ is the minimum possible number
of proper vertex colorings of $G$ so that for every pair $u,v$ of
non-adjacent vertices there is at least one coloring in which $u$ and $v$ have the
same color. What is the product dimension $Q(s,r)$ of the vertex disjoint
union of $r$ cliques, each of size $s$?  
Lov\'asz, Ne\v{s}et\v{r}il and Pultr proved in 1980 that 
for $s=2$ it is $(1+o(1)) \log_2 r$ and raised the problem of 
estimating this function for larger values of $s$.  We show that for every fixed $s$, the answer is still $(1+o(1)) \log_2 r$ 
where the $o(1)$ term tends to $0$ as $r$ tends to infinity, but the problem of determining the asymptotic behavior of $Q(s,r)$ when $s$ and $r$ grow together remains open.  The proof combines linear algebraic tools with the method of Gargano, K\"orner, and Vaccaro on Sperner capacities of directed graphs.
\end{abstract}

\section{Introduction}

The product dimension of a graph $G=(V,E)$ is the minimum possible
cardinality $d$ of a collection of proper vertex colorings of $G$ such
that every pair of nonadjacent vertices have the same color in at
least one of the colorings (and so that any two distinct vertices are colored differently in some coloring). Equivalently, this is the minimum $d$ so
that one can assign to every vertex $v$ a vector 
in $Z^d$, so that two vertices are adjacent if and only
if the corresponding vectors differ in all coordinates (and so that no two distinct vertices are assigned the same vector). If $G$ does not contain two distinct non-adjacent vertices with the same neighborhoods, as will be the case in this paper, we can take the parenthetical distinctness conditions for granted. The product
dimension is also the minimum number of complete graphs so that $G$
is an induced subgraph of their tensor product, 
where the tensor product of graphs $H_1, \ldots ,H_d$ is the 
graph whose vertex set is the cartesian product of the vertex sets of the graphs $H_i$, and two vertices $(u_1,u_2, \ldots ,u_d)$ and $(v_1,v_2, \ldots ,v_d)$ are adjacent iff $u_i$ is adjacent (in $H_i$)  to $v_i$ for all $1 \leq i \leq d$.
Yet another equivalent definition is the minimum number of subgraphs of the
complement $\overline{G}$ of $G$ so that each subgraph is a vertex disjoint 
union of cliques, and every edge of $\overline{G}$ belongs to at least one of the subgraphs (and also every pair of distinct vertices are not adjacent in some subgraph).

For positive integers $s,r \geq 2$ let $K_s(r)$ denote the graph
consisting of $r$ pairwise vertex disjoint copies of the complete
graph $K_s$.  Any two non-adjacent vertices of this graph have different neighborhoods. Let $Q(s,r)$ denote the product dimension
of this graph. Lov\'asz, Ne\v{s}et\v{r}il and Pultr \cite{LNP}
(see also \cite{Al1})
proved that $Q(2,r)=\lceil \log_2 (2r) \rceil$.
The proof of the upper bound  is simple. If $q=\lceil \log_2 (2r) 
\rceil$ then 
$2^q \geq 2r$. Hence one can assign
distinct binary vectors of length $q$ to the $2r$ vertices of $K_2(r)$ 
so that the vectors assigned to each pair of adjacent vertices are
antipodal, i.e. they differ in all coordinates. It is easy to check
that two vertices are adjacent if and only if the corresponding vectors
differ in all coordinates, showing that $Q(2,r) \leq q.$

The lower bound is proved in \cite{LNP} by a linear algebra
argument, and the proof given in \cite{Al1} applies exterior
algebra. There is yet another (similar) short proof that proceeds
by assigning to each vertex of $K_s(r)$ a multilinear polynomial 
in $x_1,x_2, \ldots ,x_q$ that depends on the coloring used, and by showing
that these polynomials are linearly independent. 
As mentioned in the abstract, 
Lov\'asz, Ne\v{s}et\v{r}il and Pultr \cite{LNP} raised the problem
of estimating $Q(s,r)$ for larger values of $s$. 
More recently, Kleinberg and Weinberg considered the same problem, motivated by  the investigation of prophet inequalities for intersection of matroids \cite{KW}. In this paper, we determine
the asymptotic behavior of $Q(s,r)$ for any fixed $s \geq 2$ and large $r$.
\begin{theo}
\label{t11}
For every fixed $s$, $Q(s,r) = (1+o(1)) \log_2 r$, where the 
$o(1)$-term tends to $0$ as $r$ tends to infinity.
\end{theo}
The main tool in the proof is the method of Gargano, K\"orner and
Vaccaro in their work on Sperner capacities \cite{GKV}. 
For completeness, and
since we are interested
in the behavior of the $o(1)$-term in the theorem above, we
describe a variant of the method as needed here, in a combinatorial 
way that avoids any application of information theoretic techniques. 
The proof is based on what we call here $Z_s$-covering families of
vectors. 

Let $Z_s$ denote the ring of integers  modulo $s$. For a subset
$A \subset Z_s$ and a vector $v=(v_1,v_2, \ldots ,v_q) \in Z_s^q$, 
we say that $v$ is $A$-covering if for every $a \in A$ there is an
$1 \leq i \leq q$ so that $v_i=a$. The vector $v$ is covering if it
is $Z_s$-covering. A family $\FF \subset Z_s^q$ is $A$-covering if
for every ordered pair of  distinct vectors $u,v \in \FF$, the difference
$u-v$ is $A$-covering. $\FF$  is covering if it is $Z_s$-covering.
Therefore, a family $\FF$ of vectors  in $Z_s^q$ is 
covering if every element of $Z_s$ 
appears in at least one coordinate of 
the  difference between any two distinct vectors in the
family. Let $R(s,q)$ denote the maximum possible cardinality of
a covering family of vectors in $Z_s^q$. The following simple
statement describes the connection between $Q(s,r)$ and $R(s,q)$.
\begin{prop}
\label{p12}
If $R(s,q) \geq r$ then $Q(s,r) \leq q$.
\end{prop}
Note that by definition $R(s,q)=1$ for all $q<s$. 
Our main result about $R(s,q)$ is the following.
\begin{theo}
\label{t13}
\begin{enumerate}
\item
For every $q \geq s \geq 2$, $R(s,q) \leq 2^{q-1}$. Equality
holds for $s=2$.
\item
For every fixed $s$, $R(s,q) \geq (2-o(1))^q$, where the
$o(1)$-term tends to $0$ as $q \to \infty$.
\end{enumerate}
\end{theo}

The rest of this paper is organized as follows. Section 2 contains
the proof of Proposition \ref{p12} and that of a simple combinatorial lemma.
In Section 3 we present the proof of Theorem \ref{t13} and note that
in view of Proposition \ref{p12}
it implies Theorem \ref{t11}. The proof supplies
better estimates for prime values of $s$, and we thus first present
the proof for this special case (which suffices to deduce the assertion of
Theorem \ref{t11} for every $s$) and then 
describe briefly the 
proof for general $s$. The final Section 4 contains some 
concluding remarks and open problems, including some (modest) estimates for
$R(s,q)$ when $q$ is not much larger than $s$.

To simplify the presentation we omit, throughout the paper, all 
floor and ceiling signs whenever these are not crucial. Thus, we ignore all divisibility issues.

\section{Preliminaries}

We first prove Proposition \ref{p12}.  If $R(s,q) \ge r$, then there 
exists a matrix of elements in $Z_s$ with $r$ rows and $q$ columns so 
that the difference of any two rows is covering.  We will use the $q$ 
columns of this matrix to find $q$ graphs, each being 
a disjoint union of cliques, that cover the complement of $K_s(r)$.  This complement is a complete multipartite graph with $r$ parts of size $s$.  Label the vertices of this graph with elements of $Z_s$ so that in each part all labels are used exactly once. We will associate each row of the matrix to a part and each column to a vertex disjoint union of cliques.  
For a column $(a_1, \cdots, a_r)^T$, consider the following graph.  For each $0 \le k<s$ take the $k+a_i$th vertex (taken modulo $s$) from the $i$th part of size $s$ and take the union of the $s$ cliques obtained as $k$ ranges between $0$ and $s-1$.  This is clearly a union of $s$ vertex disjoint cliques.  
Now, suppose we have some two vertices of the graph we are trying to cover in different parts, say the $\ell+m$th vertex of part $i$ 
and the $\ell$th vertex of part $j$ for some $1 \le i<j \le q$ 
and some $\ell, m \in Z_s$.  Then the difference of the $i$th and $j$th rows contains an $m$ in some column $(a_1, \cdots, a_r)^T$, so that $a_i-a_j \equiv m$ in $Z_s$, and then the disjoint union of cliques corresponding to this column will cover our desired edge.

One can also phrase the proof using the proper coloring definition of $Q(s,r)$.  Say we are given a matrix of $r$ vectors $\{v_1, \cdots, v_r\}$ over $Z_s^q$ which are a $Z_s$-covering family.  We can associate the vector $v_i+(j, \cdots, j)$ to the $j$th vertex in the $i$th clique of $K_s(r).$  These vectors are all distinct for different vertices, because if $v_i+(j, \cdots, j)=v_{i'}+(j', \cdots, j')$, then $v_i-v_{i'}$ is not covering, so $i=i'$ and $j=j'$.  Now, we define $q$ colorings of $K_s(r)$ so that if a vertex $x$ is associated to $(c_1, \cdots, c_q)$, it is colored with $c_k$ in the $k$th coloring.  These colorings are proper, because if $x$ and $y$ are the $j$th and $j'$th vertex in the $i$th clique for $j \neq j'$, then their associated vectors will have a difference $(j-j', \cdots, j-j') \neq \widehat{0}$.  Now, say we are given two distinct non-adjacent vertices, say the $j$th vertex in the $i$th clique and the $j'$th vertex in the $i'$th clique, where $i \neq i'$.  Then $v_i+(j, \cdots, j)$ and $v_{i'}+(j', \cdots, j')$ will share a coordinate; $v_i-v_{i'}$ is covering and thus will be $j'-j$ in some coordinate.

Next, we need the following simple lemma.
\begin{lemma}
\label{l21}
Let $H$ be a bipartite graph with classes of vertices $A_1$, $A_2$ where
$|A_1|=n_1$, $|A_2|=n_2$, each vertex of $A_1$ has degree $d_1$, and each vertex of $A_2$ has degree $d_2$. Furthermore 
suppose that $d_2 \ge \log(2n_2)$.  Then there is a union of 
vertex-disjoint
stars with centers in $A_1$, each star having at least
$\frac{d_1}{4 \log (2n_2)}$ leaves, such that all vertices of
$A_2$ are leaves.
\end{lemma}

\begin{proof}
	Define a random subset $S$ of $A_1$ by choosing each vertex of $A_1$ to be in $S$ with probability $p=\frac{\log(2n_2)}{d_2}$ uniformly and independently.  We claim that with positive probability, each vertex of $A_2$ has between $1$ and $4\log(2n_2)$ neighbors in $S$.  
The proof is a simple union bound; a fixed vertex $v \in A_2$ has probability $$(1-p)^{d_2}< e^{-pd_2}=\frac{1}{2n_2}$$ of having no neighbors in $S$, and probability at most $$\dbinom{d_2}{4\log(2n_2)}p^{4\log(2n_2)} \le \left( \frac{ped_2}{4\log(2n_2)} \right)^{4\log(2n_2)}=\left(\frac{e}{4}\right)^{4\log(2n_2)}<\frac{1}{2n_2}$$ of 
having more than $4\log(2n_2)$ neighbors, proving the claim.  
Fix an $S$ with this property.
	
We now finish the proof of the lemma by an application of Hall's theorem.  
For all $S' \subseteq S$, let $N(S')$ denote the set of all neighbors of
$S'$ and let $e(S',A_2)$ denote the number of edges from $S'$ to
$A_2$. Then
$|N(S')| \ge \frac{e(S',A_2)}{4\log(2n_2)}
=\frac{d_1}{4\log(2n_2)}|S'|$.  
Hence every subset of $S$ 
expands by a factor of at least 
$\frac{d_1}{4\log(2n_2)}$.  
Thus by Hall's theorem, there is a 
union of disjoint stars whose centers are exactly 
the vertices of $S$, each having at least $\frac{d_1}{4\log(2n_2)}$ leaves.  
Every remaining vertex of $A_2$ is adjacent to some vertex in $S$, 
so we can simply add it to an existing star.  \end{proof}

\section{Covering families}

\subsection{The upper bound}

The following proposition implies the assertion of Theorem \ref{t13},
part 1.
\begin{prop}
\label{p31}
Fix $s \geq 2$, and let $\FF \subset Z_s^q$ be a $\{0,1\}$-covering
family of vectors. Then $|\FF| \leq 2^{q-1}$. For $s=2$ equality holds.
\end{prop}
\begin{proof}
Put $m =|\FF|$. Let $p$ be a prime divisor of $s$ and consider the 
vectors in $\FF$ as vectors in $Z_p^q$ by reducing their coordinates modulo
$p$. Note that these vectors form a $\{0,1\}$-covering family over $Z_p$, and so are distinct. Let
$v_i=(v_{i1},v_{i2}, \ldots ,v_{iq})$, $(1 \leq i 
\leq m)$ be the vectors in $\FF$ (considered as elements of 
$Z_p^q$). 

For each $1\leq i \leq m$ define two polynomials $P_i,Q_i$ in
the variables $x_1,x_2, \ldots ,x_q$ over $Z_p$ as follows.
$$
P_i(x_1, \ldots ,x_q)=\prod_{j=1}^q (x_j -v_{ij}),~~
Q_i(x_1, \ldots ,x_q)=\prod_{j=1}^q (x_j -v_{ij}-1).
$$
It is not difficult to check that for every $i$, $Q_i(v_i) \neq 0$
and $P_i(v_i)=0$. In addition, for every 
$1 \leq i \neq i' \leq m$, $P_{i'}(v_i) =0$ (as there is a coordinate
$j$ for which $v_{ij}-v_{i'j}=0$) and $Q_{i'}(v_i)=0$ (as there
is a $j$ so that $v_{ij}-v_{i'j}=1$).

Similar reasoning gives that for the vectors $v_i+J$, where $J$ is the all
$1$-vector of length $q$, $P_i(v_i+J) \neq 0$, $Q_i(v_i+J)=0$, and
for every $i' \neq i$, 
$P_ {i'}(v_i+J) =Q_{i'} (v_i +J)=0$. Therefore, for each member of
the collection
of $2m$ polynomials $\{P_i,Q_i: 1 \leq i \leq m\}$ there is an 
assignment of values of the variables in which this member is nonzero
and all others vanish. This easily implies that the set of $2m$ polynomials
$P_i,Q_i$ is linearly independent in $Z_p$, 
and as each of its members lies in the space of multilinear
polynomials with the $m$ variables $x_j$, the number, $2m$,
of these polynomials is at most the dimension of this space which is
$2^q$. It follows that $|\FF| =m \leq 2^{q-1}$, as needed. For $s=2$
the family of all binary vectors in which the first coordinate is $1$
is $\{0,1\}$-covering, showing that $R(2,q)  \geq 2^{q-1}$ and 
completing the proof. $\Box$

\end{proof}

\subsection{Prime $s$} For prime $s \ge 3$, we will prove that $R(s,q) 
\ge (2-o(1))^q$ where the $o(1)$-term tends to $0$ as $q \to \infty$.  
The crux of the proof is a Markov chain argument from \cite{GKV}, which we will iterate $O(\log s)$ times.

A \emph{balanced word} over $Z_s^q$ is a word containing the letters $1$ through $s-1$ an equal number of times.  A \emph{special balanced word} is a balanced word such that the first $\frac{q}{(s-1)/2}$ letters are $1$ and $2$ in some order, the next $\frac{q}{(s-1)/2}$ letters are $3$ and $4$ in some order, and so on.  
Construct a bipartite graph between the set $A_2$ of balanced words 
$w$ over $Z_s^q$ and the set $A_1$ of permutations $\pi$ on $q$ elements defined as follows: $w$ and $\pi$ are adjacent if and only if $\pi(w)$ is a special balanced word.  
By symmetry all $n_1=q!$ vertices in $A_1$ have the same degree $d_1$, and all vertices in $A_2$ have the same degree $d_2$.  We have $$n_2=|A_2|=\dbinom{q}{q/(s-1), \cdots, q/(s-1)} \le (s-1)^q$$ and $d_1$ is the total number of special balanced words, 
so $$d_1=\dbinom{2q/(s-1)}{q/(s-1)}^{(s-1)/2}>\frac{2^q}{q^{s/2}}.$$  Furthermore, $d_2=\frac{n_1}{n_2}d_1=\left((\frac{q}{s-1})!\right)^{s-1}d_1 \ge \log(2n_2)$ so by Lemma \ref{l21} there exists a way to map balanced words to some set $T$ of permutations $\pi$ of $q$ elements, so that each balanced word is 
associated to exactly one permutation, and each permutation in $T$ is associated to at least $\frac{d_1}{4\log(2n_2)}>\frac{2^q/q^{s/2}}{4\log(2(s-1)^q)}>\frac{2^q}{q^s}$ balanced words.  
Thus, we can partition the balanced words into sets $S_1, S_2, \cdots$ so that for each $S_i$ we have $|S_i|>\frac{2^q}{q^s}$ and for all $i$ there exists $\pi_i$ so that $\pi_i(s_i)$ is a special balanced word for all $s_i \in S_i$.

Given all of the special balanced words of length $q$, any two of them have a difference vector which covers $\{\pm 1\}$.  The idea of the proof will be to amplify this set $\{-1,1\}$, first to $\{-\alpha,-1,1,\alpha\}$ for a primitive root $\alpha$ modulo $s$, and after $r$ steps to $\{\pm \alpha^b\}$ for $0 \le b<2^r$.  At each stage, the number of vectors will be $(2-o(1))^L$ where $L$ is the length of the vectors.  Thus after $O(\log s)$ steps we will have a set of vectors that is $Z_s^{*}$ covering.  We can then add an extra coordinate of $0$ to all of the vectors to make them $Z_s$-covering.

We describe the first step of this iteration in detail.  Fix a primitive root $\alpha$ modulo $s$, which will be constant throughout the steps.  Also fix $n=\frac{100}{\epsilon}(\log(s))^2$, which will again be constant throughout the steps.  
Initially we set $q=q_0=100(s^2-1)$, ensuring it is divisible by
$s-1$.  
We will construct words of length $q_1=qn$ by stringing together balanced words of length $q$ in a specific way.  If $x_i \in S_j$, then force $x_{i+1} \in \alpha S_j$; here we mean that if we take $x_{i+1}$ and multiply its letters by $\alpha^{-1}$ pointwise, the result will be in $S_j$.  
Consider all vectors of length $qn$ constructed according to this rule, by concatenating $n$ balanced words of length $q$ in this way.  
There are more than $n_2(\frac{2^q}{q^s})^{n-1}$ such words, 
because at each stage other than the first we must pick 
$x_{i+1}$ so that $x_i \in \alpha S_j$ for some $j$, and thus there 
are more than $\frac{2^q}{q^s}$ choices for $x_{i+1}$. 
We will make $x_1, x_n$ identical over all such words; 
this costs us a factor of $n_2^2$ and thus we now 
have more than $\frac{1}{n_2}(\frac{2^q}{q^s})^{n-1}$ such words.  
Using that $n_2<s^q$, we can find more than $\frac{2^{q(n-1)}}{q^{s(n-1)}s^q}=(2-o(1))^{qn}$ words of length $qn$ of the form $x_1x_2 \cdots x_n$ so that $x_1, x_n$ are fixed, where $o(1)$ is as $\epsilon$ goes to $0$ and $n$ goes to infinity.  We now note that for any two different words of this form $x_1 \cdots x_n$ and $x'_1 \cdots x'_n$, 
with $x_1=x'_1, x_n=x'_n$, there is some minimal $i \ge 2$ so that $x_i \neq x'_i$.  But then there is some $k$ so that $x_{i-1}=x'_{i-1} \in S_k$, and that means $x_i, x'_i \in \alpha S_k$.  
Because $x_i \neq x'_i$, we have $\alpha^{-1}x_i \neq \alpha^{-1}x'_i$, and both $\alpha^{-1}x_i$  and $\alpha^{-1}x'_i$ are in $S_k$.  It follows that $\alpha^{-1}x_i-\alpha^{-1}x'_i$ covers $\{\pm 1\}$ and so $x_i-x'_i$ covers $\{\pm \alpha\}$.  
Similarly, there is some maximal $i<n$ so that $x_i \neq x'_i$.  Then $x_{i+1}=x'_{i+1} \in \alpha S_j$ for some $j$, so $x_i \in S_j$, $x'_i \in S_j$.  As $x_i \neq x'_i$, it follows that $x_i-x'_i$ must cover $\{\pm 1\}$ as coordinates.  
When $s=5$, setting $\alpha=2$ and applying this construction already 
gives a covering family for $\mathbb{Z}_s^{*}=\mathbb{Z}_5^*$ 
with $(2-o_N(1))^N$ vectors of length $N$.  In the general case we iterate this argument to find $(2-o_N(1))^N$ vectors of length $N$, so that after the $r$th iteration the vectors we get cover $\{\pm \alpha^b\}$ for all $0 \le b<2^r$.  We describe how to do this inductively.  

After $r$ iterations, we find for some $q=q_r$ (that depends on $s$) 
a family of $M^q=(2-o_q(1))^q$ (balanced) vectors over $Z_s^q$ which covers $\pm \alpha^b$ for all $0 \le b<2^r$.  We call these vectors $y_1, \cdots, y_{M^q}$ and let $Y=\{y_1, \cdots, y_{M^q}\}$.  We repeat the above argument.  
The $y_i$ play the role of the special balanced words.  
Again we construct a bipartite graph.  
On one side there is the set $A_2$ of all balanced words of length $q$ and on the 
other side there is $A_1$, permutations of $q$ elements.  
We have $\pi$ is adjacent to a (balanced) vector $y$ if and only 
if $\pi(y) \in Y$.  Again $n_2=|A_2| \le (s-1)^q$ but now $d_1=M^q$.  
It can easily be verified that $d_2 \ge \log(2n_2)$.  Thus by Lemma \ref{l21} the balanced words of length $q$ will be split into sets $S_1, S_2, \cdots$ so that all $S_i$ satisfy $|S_i| \ge \frac{d_1}{4\log(2n_2)} \ge \frac{M^q}{4+4q\log(s-1)}=(2-o_q(1))^q$, and furthermore for each $S_i$, there exists a permutation $\pi_i$ so that for all $s_i \in S_i$, $\pi_i(S_i) \in Y$.  Now we will again construct a Markov chain.  
Consider all words of length $qn=q_{r+1}$ consisting of $n$ balanced words of length $q$ of the form $x_1x_2 \cdots x_n$, so that if $x_i \in S_j$, then $x_{i+1} \in \alpha^{2^r}S_j$.  If we have two such words $x_1x_2 \cdots x_n$ and $x'_1x'_2 \cdots x'_n$ so that $x_1=x'_1$, $x_n=x'_n$, let $i>1$ be minimal so that $x_i \neq x'_i$.  Then if $x_{i-1}=x'_{i-1} \in S_j$, $x_i, x'_i \in \alpha^{2^r}S_j$.  
This means that $x_i-x'_i$ covers $\{\pm \alpha^b\}$ modulo $s$ for $2^r \le b<2^{r+1}$.  Now, if we let $x_i$, $i<n$, be maximal so that $x_i \neq x'_i$, then $x_{i+1}=x'_{i+1} \in \alpha^{2^r}S_j$ for some $j$ and so $x_i, x'_i$ are not equal but are both in $S_j$.  
This means $x_i-x'_i$ covers $\{\pm \alpha^b\}$ modulo $s$ 
for $0 \le b<2^r$.  So indeed the family covers $\{\pm \alpha^b\}$ 
modulo $s$ 
for $0 \le b<2^{r+1}$, as long as $x_1$ and $x_n$ are fixed over the family.  We can always find 
$$
\frac{\min_{i}|S_i|^{n-1}}{\dbinom{q}{q/(s-1), \cdots, q/(s-1)}} 
\ge \frac{(2-o(1))^{qn}}{(s-1)^q}=(2-o(1))^{qn}
$$ 
such words if $n$ is large enough, where the $o(1)$ terms tend to $0$ 
as $q \to \infty$ and $n$ is sufficiently large.
We will soon see that our choice of $n=\frac{100}{\epsilon}(\log(s))^2$ 
is sufficient.  Iterating the argument $ \log_2(s) $ 
times allows us to find $(2-o(1))^q$ vectors of some length $q$ which cover $Z_s^{*}$.  
Adding a single coordinate where all vectors are $0$ gives us vectors of 
length $q+1$ that cover $Z_s$, without changing the asymptotic analysis.

With care, we can extract quantitative bounds.  We assume $\epsilon>0$ is a fixed constant and show that we only require $q=\left(\frac{O(1)}{\epsilon}(\log(s))^2\right)^{\log(s)}$ to have a $(2-\epsilon)^q$ size covering system over $Z_s^q$ for large $s$.  
Say that after $r$ iterations, we have $M^q=M_r^q$ vectors of length 
$q=q_r$ for some $M \le 2$ (which is nearly $2$).  
Then $\min_{i}S_i \ge \frac{M^q}{5q\log(s)}$  and thus we find at least $$\frac{\min_{i}|S_i|^{n-1}}{\dbinom{q}{q/(s-1), \cdots, q/(s-1)}} \ge \frac{M^{qn}}{M^q(5q\log(s))^ns^q} \ge \frac{M^{qn}}{(2s)^q(5q\log(s))^n}$$ $$=\left( \frac{M}{(2s)^{1/n}q^{1/q}(5\log(s))^{1/q}} \right)^{qn}=M_{r+1}^{q_{r+1}}$$ vectors of length $qn=q_{r+1}$.  
Recall that $n=\frac{100}{\epsilon}(\log(s))^2$ and $q_0=(1+o(1))100s^2$, where $o(1)$ is as $s$ goes to infinity, 
so when we iterate the Markov chain argument $\log_2(s)$ times, we lose a factor of at most 
$$
\frac{M_0}{M_{\log_2(s)}} \le (2s)^{2\log(s)/n}
\left(\prod_{j=0}^{\infty}(q_0n^j)^{1/q_0n^j}\right)(5\log(s))^{1/q_0}
(5\log(s))^{2\log(s)/n}
$$ 
$$
\le (10s\log(s))^{2\log(s)/n}2^{10\log(s)/s^2}(5\log(s))^{1/100s^2}.
$$ 
In the last inequality here
the bound on the second term holds because the relevant infinite product  
is at most $(q_0^{1/q_0})^{100}<2^{10 \log s/s^2}$.
For large $s$ this is easily seen to be smaller than, say,
$1+\epsilon/4$. Since for large $s$, $M_0>2-\epsilon/2$, for 
$s>s_0(\epsilon)$ we get $M_{\log_2(s)} >2-\epsilon$.
Thus at the end we have 
$q=100s^2\left(\frac{O(1)}{\epsilon}(\log(s))^2\right)^{\log_2(s)}$ 
and at least $(2-\epsilon)^q$ vectors of length $q$ which cover $Z_s$.  
One can easily modify the argument to work for any larger $q$, or simply 
use the super-multiplicative property of $R(s,q)$ (see the 
beginning of Section 4)
to conclude, 
taking $\epsilon=1/\log s$,  that for every large $s$ and for all
$q>s^{(3+o(1)) \log \log s}$, $R(s,q) > (2-\frac{1}{\log s})^q$.
This completes the proof of Theorem \ref{t13}, part 2, for prime $s$.

\subsection{General $s$}

Given an arbitrary fixed integer $s>2$, we now show how to find $(2-o(1))^q$ vectors in $Z_s^q$ which form a $Z_s$-covering family; this shows 
$R(s,q) \ge (2-o(1))^q$ even for composite $s$.  
The general strategy is similar to our strategy in the previous section, except now we work over $Z$, and the set we are covering does not grow 
so quickly. As before, we are not concerned with the vectors 
covering $0$, because we can simply add an extra 
coordinate to deal with it.   Since the argument is very similar to
the one described in the previous subsection, we only provide a brief
description omitting some of the formal details.

We prove the stronger statement that for any fixed $s$, we can find a $(2-o_q(1))^q$ size family over $Z$ that covers $[-s,s]$, i.e. the difference of any two vectors contains all integers between $-s$ and $s$ as coordinates.  Let $S$ be the least common multiple of the first $s$ positive integers.  We will assume without loss of generality that $2S \mid q$. 

At the first step of our iteration, we consider vectors over $Z^q$ 
with an equal number of each element of $[2S]$ as coordinates so that the first $\frac{q}{S}$ coordinates are $1$s and $2$s in some order, 
the next $\frac{q}{S}$ are an equal number of $3$s and $4$s, and so on.  
We can find $(2-o_q(1))^q$ of these and they cover $\{ \pm 1\}$.  Furthermore, they have the property that for any ordered pair of these vectors, there is a coordinate in which the first has an even integer $2k$ and the second has the odd integer $2k-1$ for some $1 \le k \le S$.  This property is crucial and maintained throughout our iterations.

Now we describe the $m$th step of our iteration, for $2 \le m \le s$.  Define a bijection $f=f_m$ on $[2S]$ so that for all integers $1 \le k \le S$, $f(2k)=f(2k-1)+m$.  
We can do this for instance by setting $f(1)=1, f(2)=m+1, f(3)=2, f(4)=m+2, \cdots, f(2m-1)=m, f(2m)=2m$ and then set $f(x)=f(x-2m)+2m$ for $x>2m$ as long as $x \le 2S$.

We then apply the same Markov chain argument as before, defining 
sets $S_i$ and constructing words $x_1x_2 \cdots$ 
starting and ending at the same vectors.  
Now, however, when $x_i$ is in some set $S_j$, instead of 
demanding $x_{i+1} \in \alpha S_j$ we require
$x_{i+1} \in f(S_j)$, meaning that if we apply $f^{-1}$ to each coordinate of $x_{i+1}$, the result will be in $S_j$.  
Looking at the first place where two vectors differ gives us a difference of $\pm m$.  Looking at the last place where they differ shows that the crucial property is preserved, and also that the differences $\pm 1, \cdots, \pm (m-1)$ are retained.  

Thus after $s=O(1)$ iterations, this algorithm produces $(2-o_q(1))^q$ 
vectors of length $q$ which cover all of $[-s,s]$, after we add 
an extra coordinate to deal with covering $0$. This completes the proof
of Theorem \ref{t13}. 

\section{Concluding remarks and open problems} 

A natural open problem is to study the functions 
$R(s,q)$ and $Q(s,r)$ in general.  There are several
simple properties that $R(s,q)$ satisfies.  We know that $R(s,q)$ is 
(weakly) increasing in $q$, because to create $r$ covering 
vectors for $Z_s$ of length $q'>q$, we can take $r$ vectors for $Z_s$ of length $q$ and pad them with $q'-q$ zeroes at the end.  
Furthermore, we know that $R(s,q)$ is super-multiplicative, i.e. $R(s,q_1+q_2) \ge R(s,q_1)R(s,q_2)$, 
because if $m=R(s,q_1)$, $n=R(s,q_2)$ then we can find 
vectors $v_1, \cdots, v_m$ of length $q_1$ and $w_1, \cdots, w_n$ of 
length $q_2$ that form covering families. The $mn$ vectors $v_iw_j$ 
obtained by concatenating $v_i$ and $w_j$ 
are clearly a covering family of length $q_1+q_2$.

For $q<s$, we have $R(s,q)=1$, because of course we can take a single vector in $Z_s^q$, but if we take two then their difference can only 
cover a set of size $q$ and cannot cover $Z_s$.

The next natural question is studying the value of $R(s,s)$.

\begin{prop}
\label{p41}
$R(s,s) \le s$, and $R(s,s) \ge p$ where $p$ is the smallest prime 
factor of $s$.  When $p=2$ this is tight, that is,
if $s$ is even, then $R(s,s)=2$.
\end{prop}  

\begin{proof}
	For the lower bound, for each $0 \le a<p$ we have a vector $(0,a,2a, \cdots, (s-1)a)$ reduced modulo $s$.  This is a covering system for $Z_s$, since all positive integers smaller than $p$ are relatively prime to $s$.
	
	For the upper bound, assume there was a covering family in $Z_s^s$ with $s+1$ vectors, so that the difference of any two of these vectors has all values of $Z_s$ exactly once.  
By the pigeonhole principle, there exist 
two vectors $v$ and $w$ so that the difference of 
their first and second coordinates is the same.  But then $v-w$ has the same value in its first and second coordinate, a contradiction.
	
	When $s$ is even, $R(s,s) \ge 2$ as $2$ is the least prime factor of $s$.  If $R(s,s) \ge 3$ then there exist $3$ 
vectors in $\mathbb{Z}_s^s$, $(a_1, \cdots, a_s), (b_1, \cdots, b_s), (c_1, \cdots, c_s)$, which are covering.  
But then $\sum (a_i-b_i) \equiv \sum_{j=0}^{s-1}j \equiv (s-1)\frac{s}{2} \equiv \frac{s}{2}$ modulo $s$, and similarly $\sum (b_i-c_i) \equiv \sum (c_i-a_i) \equiv \frac{s}{2}$ modulo $s$.  Hence, these three sums are all $\frac{s}{2}$ modulo $s$, so they must add up to $\frac{3s}{2} \equiv \frac{s}{2}$ modulo $s$.  But they add up to $0$, so this is a contradiction.
\end{proof}

Nonetheless, the problem of determining $R(s,s)$ for all $s$ 
remains open, and the lower bound in the last proposition
is not tight.  
A computer search gives that $R(15,15) \ge 4$ (\cite{L}).  One example is the $4$ vectors given by the rows of the matrix $$\begin{pmatrix} 
0 & 0 & 0 & 0 & 0 & 0 & 0 & 0 & 0 & 0 & 0 & 0 & 0 & 0 & 0 \\ 
0 & 1 & 2 & 3 & 4 & 5 & 6 & 7 & 8 & 9 & 10 & 11 & 12 & 13 & 14 \\  
0 & 2 & 1 & 5 & 7 & 9 & 12 & 14 & 13 & 3 & 6 & 4 & 10 & 8 & 11 \\
0 & 3 & 9 & 1 & 10 & 14 & 7 & 11 & 4 & 12 & 5 & 8 & 2 & 6 & 13 \end{pmatrix} $$ 
which are covering for $Z_{15}$.

For $q$ which is only a little bigger 
than $s$, we can prove a reasonable upper bound, 
using essentially the same observation applied in the 
proof that $R(s,s) \le s$.

\begin{prop}
	If $2(q-s)^2<s-1$, then $R(s,q) \le s+2(q-s)+2(q-s)^2$. 
\end{prop}

\begin{proof}
	Assume a contradiction for some $q>s$.  
By the definition of $R(s,q)$, for any $r \le R(s,q)$ there exists a matrix of $r$ rows and $q$ columns so that the difference of any two rows contains all values modulo $s$; in particular there exists such a matrix for $r=s+1+2(q-s)+2(q-s)^2$.  
We will double count the number of pairs of rows and columns $r_i, r_j, c_k, c_{\ell}$ with the following property: the $2 \times 2$ submatrix formed by $r_i, r_j, c_k, c_{\ell}$ has sums of opposite corners equal.
	
	For any two rows $r_i, r_j$, we examine their difference.  This is a vector over $Z_s$ with $q$ entries containing all elements modulo $s$, so at most $\dbinom{q-s+1}{2}$ pairs of its entries can be equal.  
This means that for the pair of rows $r_i, r_j$, there exist at most $\dbinom{q-s+1}{2}$ pairs of columns $c_k, c_{\ell}$ that satisfy our property.  
So the total number of such pairs as $1 \le i,j \le r$, $1 \le k, \ell \le q$ range is at most $\dbinom{r}{2}\dbinom{q-s+1}{2}$.  However, if we instead fix $c_k, c_{\ell}$ we see that the difference of these two columns is a vector of length $r$ over $Z_s$ which has at least $r-s$ pairs of equal 
elements. This means that 
there are at least $\dbinom{q}{2}(r-s)$ $r_i, r_j, c_k, c_{\ell}$ with our property.  
Hence 
$$
\dbinom{q}{2}(r-s) \le \dbinom{r}{2}\dbinom{q-s+1}{2}
$$ and so $2q(q-1)(r-s) \le r(r-1)(q-s+1)(q-s).$  
By our assumption, $2q>r$ and so $2(q-1) \ge r-1$.  Hence $2q(q-1)(r-s) \le r(r-1)(q-s+1)(q-s) \le 4q(q-1)(q-s+1)(q-s).$  Thus $r-s \le 2(q-s+1)(q-s)=2(q-s)+2(q-s)^2$ and so $r \le s+2(q-s)+2(q-s)^2$, a contradiction.  \end{proof}
		
Note that this argument gives a bound only when $q=s+O(\sqrt{s})$, 
and the bound is $s+O((q-s)^2)$.  We believe that this is not tight. 
	
There is a nontrivial (though weak) 
bound which holds in the regime $q=s+\omega(\sqrt{s})$ as long as $q<Cs$ 
for some fixed constant $C< \log_2 e$.  
The natural open problem here is to study cases when $q$ is larger but still fairly small, for instance if $q=1.5s$, $q=s\log(s)$, or $q=s^2$.  Furthermore, it would be interesting to figure out how large $q$ has to be so that $R(s,q)>(2-\epsilon)^s$ for a small positive constant $\epsilon$; the following argument shows that if $\epsilon$ is small enough, then $\frac{q}{s}$ must be bigger than an absolute constant which is above $1.44$, but we believe this is far from tight.

\begin{prop} 
If $1 \le \frac{q}{s}<C<\log_2(e)$, 
then $R(s,q) \le (2-\varepsilon)^q$ for $\varepsilon=\varepsilon(C)>0$.
\end{prop}

\begin{proof}
Say that $R(s,q) \ge r$ so there are $r$ vectors in $Z_s^q$ which are covering.  Place them in a matrix with $r$ rows and $q$ columns.  Now, for each column select uniformly an element of $Z_s$, and add it to all the elements of that column, reducing modulo $s$.  This does not change the covering property.  	
After having done this, replace all entries of this matrix by their reduction modulo $2$.  Given two entries of a matrix in the same column and different rows, if initially they differed by $k$ modulo $s$ they now have some probability $p_k$ of being equal that depends only on $k$, 
which is the probability that if $x \in Z_s$ is chosen randomly 
and uniformly 
then the reductions of $x,x+k$ modulo $s$ have the same parity.  One can easily check that $p_0=1$, $p_1=1/s$, $p_2=(s-2)/s$, and so on, 
so that $\prod_{i \in Z_s}p_i=e^{-s(1+o(1))}$ by Stirling's formula.  
Furthermore, if we take this product over all $i \in Z_s$ except 
for a subset of size 
$2cs$ for a small constant $c$, we have a bound of $e^{-s(1-\epsilon)}$, where $\epsilon$ is a small positive constant that tends to $0$ with $c$.

Let $K=cs$.  For any pair of rows, the probability their colors match in all but at most $2K$ places is at most 
$e^{-s(1-\epsilon')}$ for some $\epsilon'$ that tends to $0$ with $c$, 
by a union bound over all 
the at most $2\dbinom{q}{2K}$ 
possible sets of places where they do not match.  Thus if we define
a graph on our $r$ vectors where two are adjacent iff their values 
upon reduction modulo $2$ match in all but at most $2K$ places, 
this graph will in expectation have edge density $e^{-s(1-\epsilon')}$ 
and thus for some fixed choice of the random shifts
will have at most $r^2e^{-s(1-\epsilon')}$ edges.  Fixing these shifts
one can remove less than 
$r^2e^{-s(1-\epsilon')}$ vertices from this graph and get an 
independent set. But this set
must be of size at most $(2-\delta)^q$ for some $\delta=\delta(c) > 0$,
because this size is bounded by the cardinality 
of a family of disjoint Hamming balls in $\{0,1\}^q$, each of 
radius $K=cs$.  Thus we have $r \le r^2e^{-s(1-\epsilon')}+(2-\delta)^q$, 
and the same inequality holds for any $r'<r$ by applying the same
reasoning to a set of $r'$ of our vectors.

Now if $q<Cs$ for some fixed  $C< \log_2 e$ then setting $q > q_0(\epsilon',C)$, $r'=3 (2-\delta)^q$ violates the 
last inequality (with $r$ replaced by $r'$) since for this value
of $r'$ and large $q$, $(r')^2 e^{-s(1-\epsilon')}<r'/2$ and
$(2-\delta)^q <r'/2$ so their sum is smaller than $r'$.
This establishes the assertion of the proposition.
\end{proof}

On the opposite extreme, one may ask 
what happens if $s \ge 3$ is a small fixed positive integer 
and $q$ grows.  We know by Proposition~\ref{p31} that $R(2,q)=2^{q-1}$, 
so it is natural to ask what happens when $2$ is replaced by 
a larger positive integer.

\begin{conj}
For any fixed $s \ge 3$, $R(s,q)=o(2^q)$, 
and furthermore $R(s,q)=\Theta(2^q/q^{c})$ where 
$c=c(s)>0$ is a constant that depends only on $s$.
\end{conj}

Note that the vectors over $Z_3^q$ with a zero in the first coordinate and with exactly $\lfloor q/2 \rfloor$ ones and $\lceil q/2 \rceil$ zeroes are a covering system, so $R(3,q)=\Omega(2^q/\sqrt{q})$.  
When $s$ is odd the best upper bound known for $R(s,q)$ 
is $O(2^{q})$, as shown in Proposition \ref{p31}.  
When $s$ is even, if the vectors of a covering system for $Z_s^q$ are reduced modulo $2$, any two differ in at least $s/2$ places, 
so there are at most $O(2^q/q^{\lfloor \frac{s-1}{4} \rfloor})$ of them, 
as the Hamming balls of radius $\lfloor \frac{s-1}{4} \rfloor$ 
centered at these reduced vectors are pairwise disjoint.
It would be interesting to establish the above conjecture and to find the relevant constants $c$ if they indeed exist.  
In particular it seems plausible that 
$c=1/2$ when $s=3$, and if so then the lower bound for $R(3,q)$ is 
essentially optimal.  We conjecture that this is indeed the case.

\begin{conj}
	$R(3,q)=\Theta(2^q/\sqrt{q}).$
\end{conj}

In \cite{Ca} it is shown that $R(3,q) \le (\frac{1}{2}+o(1))2^q$ 
when $q$ is even and $R(3,q) \le (\frac{1}{3}+o(1))2^q$ when $q$ 
is odd, leaving this problem open.  Note that we do not even know 
how to prove the weaker claim that a $\{-1,0,1\}$ covering system 
(or equivalently just a $\{\pm 1\}$ covering system) 
over $Z$ must have size $o(2^q)$.
\vspace{0.2cm}

\noindent
Our original motivation for studying the function $R(s,q)$ here is
its connection to the product dimension $Q(s,r)$ of the disjoint union of
$r$ cliques, each of size $s$.  This connection is described in Proposition
\ref{p12}. The results here suffice to determine the asymptotic behaviour
of $Q(s,r)$ for every fixed $s$ as $r$ tends to infinity, but do not provide
tight bounds when $r$ is not much bigger than $s$. 
We conclude this short paper with several simple comments about this range
of the parameters. The first remark is that the product dimension
$Q(s,r)$ is at least $s$ for every $r \geq 2$. To see this fix a 
vertex $u$ of the first clique and observe that in every proper coloring of
the graph $K_s(r)$ of $r$ disjoint cliques, each of size $s$, there is
at most one pair $uv$ with a vertex $v$ of the second clique so that $u$
and $v$ have the same color.  As altogether there  are $s$ such pairs, 
and each one has to be monochromatic in at least one vertex coloring 
in a collection exhibiting an upper bound for the product dimension, the
number of such colorings is at least $s$.  Another comment is that
$Q(s,r)$ is clearly monotone non-decreasing in both $r$ and $s$. Therefore,
for every $r,s \geq 2$, $Q(s,r) \geq Q(2,r) = \lceil \log_2(2r) \rceil$.

A {\em transversal design} $TD(r,s)$ of order $s$ and block size $r$ (with 
multiplicity $\lambda=1$) is a set $V$ of $sr$ elements partitioned 
into $r$ pairwise disjoint groups, each of size $s$, and a collection
of blocks, each containing exactly one element of each group, so that
every pair of elements from distinct groups is contained in exactly
one block. A transeversal design is {\em resolvable}
if its blocks can be partitioned
into parallel classes where the blocks in any parallel class 
partition the set $V$.
There is a substantial amount of literature about transversal
designs, see \cite{CD}. It is not difficult to check that 
$Q(s,r)=s$ if and only if a resolvable 
$TD(r,s)$ exists and hence the known results
about resolvable 
transversal designs supply nearly precise information for the range
$r \leq s$ (it is easy to see that such a design cannot exist for 
$r>s$).  In particular, for every prime power $s$, $Q(s,s)=s$
and therefore  by the obvious monotonicity, for any prime power $s$,
$Q(s,r)=s$ for every $2 \leq r \leq s$, and if $p$ is a prime power then for
every $s,r \leq p$, $Q(s,r) \leq p$. (Note that 
Propositions \ref{p12} and \ref{p41}  also imply that 
$Q(s,s)=s$ when $s$ is
a prime.)

It is not difficult to prove that for every $s,r_1,r_2$,
\begin{equation}
\label{e999}
Q(s,r_1r_2) \leq Q(s,r_1)+Q(s,r_2). 
\end{equation}
Indeed, given the graph
$K_s(r_1r_2)$ consisting of $r_1r_2$ disjoint cliques, each of size
$s$, we can split the cliques into $r_1$ disjoint groups, each consisting
of $r_2$ cliques. Define $Q(s,r_1)$ 
proper colorings in which the cliques
in every group are colored the same, based on the system of colorings that
shows that the product dimension of $K_s(r_1)$ is $Q(s,r_1)$. Add to these
$Q(s,r_2)$ additional colorings, whose restrictions to the $r_2$ cliques
in each group are exactly the colorings showing that the product dimension
of $K_s(r_2)$ is $Q(s,r_2)$. The resulting $Q(s,r_1)+Q(s,r_2)$ colorings
establish (\ref{e999}).

The above comments together with the results in the previous sections 
provide upper and lower bounds for $Q(s,r)$ for all
$s$ and $r$, but these bounds are quite far from each other when
$r$ is much bigger than $s$ but much smaller than $2^{s^{3 \log \log s}}$. 
In particular, for $r=2^s$ the bounds we have are
$$
s \leq Q(s,2^s) \leq (1+o(1)) \frac{s^2}{\log s}.
$$
It would be interesting to close this gap.  
\vspace{0.2cm}

\noindent
{\bf Acknowledgment}
We thank J\'anos K\"orner, Robert Kleinberg and Matt Weinberg  
for helpful comments. We also thank two anonymous referees for helpful suggestions.

\end{document}